\documentclass[12pt]{amsart}
\usepackage[english]{babel}
\usepackage{amsmath,amsthm,amssymb,amsfonts}

\begin{document}

\hfuzz=6pt

\widowpenalty=10000

\newtheorem{theorem}{Theorem}[section]
\newtheorem{proposition}[theorem]{Proposition}
\newtheorem{coro}[theorem]{Corollary}
\newtheorem{lemma}[theorem]{Lemma}
\newtheorem{definition}[theorem]{Definition}
\newtheorem{assum}{Assumption}[section]
\newtheorem{example}[theorem]{Example}
\newtheorem{remark}[theorem]{Remark}

\renewcommand{\theequation}
{\thesection.\arabic{equation}}
\newcommand{\one}{1\hspace{-4.5pt}1}
\newcommand\RR{\mathbb{R}}
\newcommand\CC{\mathbb{C}}
\newcommand\NN{\mathbb{N}}
\renewcommand\Re{\operatorname{Re}}
\renewcommand\Im{\operatorname{Im}}
\newcommand\D{\mathcal{D}}
\newcommand{\disp}[1]{\displaystyle{#1}}
\newcommand{\und}[1]{\underline{#1}}
\newcommand{\norm}[1]{\left\Vert#1\right\Vert}
\newcommand{\supp}{\textrm{supp}}

\def\ang#1{\langle #1 \rangle}
\def \SL {\sqrt L}
\def \l {\lambda}
\def\HSL { H^p_{L, S_h}(X) }
\def\pl {\partial}

\title[Sharp  spectral multipliers for Hardy spaces ]{%
Sharp  spectral multipliers  for Hardy spaces associated
to  non-negative
self-adjoint  operators satisfying \\ Davies-Gaffney estimates }
\medskip

\author{Peng Chen}
\address{Peng Chen, Department of Mathematics, Macquarie University, NSW 2109, Australia}
\email{achenpeng1981@163.com} 

\begin{abstract}
We consider the abstract non-negative self-adjoint operator $L$
acting on $L^2(X)$ which satisfies Davies-Gaffney estimates and the
corresponding Hardy spaces $H^p_L(X)$. We assume that doubling
condition holds for the metric measure space $X$. We show that a
sharp H\"ormander-type spectral multiplier theorem on $H^p_L(X)$
follows from restriction type estimates and the Davies-Gaffney
estimates.
 We also describe the sharp result for the boundedness of Bochner-Riesz means on
$H^p_L(X)$.
\end{abstract}

\maketitle

\section{Introduction}
\setcounter{equation}{0} Suppose that $L$ is  a  non-negative
self-adjoint operator acting on $L^2({X,\mu})$, where $X$ is a
measure space with measure $\mu$. Then $L$ admits a spectral
resolution $E(\lambda)$ and for any bounded Borel function $F: [0,
\infty)\rightarrow {\Bbb C}$, one can define the operator
\begin{eqnarray}
F(L)=\int_0^{\infty} F(\lambda) dE(\lambda). \label{e1.4}
\end{eqnarray}
By the spectral theorem, this operator is bounded on $L^2(X)$.
Spectral multiplier theorems give sufficient conditions on   $F$ and
$L$
  which imply the boundedness of $F(L)$  on various functional
spaces defined on $X$. This is one of active topics in harmonic
analysis and has been studied extensively. We refer the reader to
 \cite{A, B, C, COSY, CS2, DeM, DOS, DP1, DY2, GHS, M, St}
   and the references therein.

\medskip

Before we state our main result we describe our basic assumptions.
We throughout assume that the considered metric measure space
$(X,d,\mu)$ with a distance $d$ and a non-negative Borel measure
$\mu$ satisfies the volume doubling condition: there exists a
constant $C>0$ such that for all $x\in X$ and for all $r>0$,
\begin{eqnarray}
V(x,2r)\leq C V(x, r)<\infty, \label{e1.1}
\end{eqnarray}
where $V(x,r)$ is the volume of the ball $B(x,r)$ centered at $x$ of
radius $r$. In particular, $X$ is a space of homogeneous type. See
for example \cite{CW}.

Note that the doubling condition \eqref{e1.1} implies that there
exist some constants $C, n>0$ such that
\begin{equation}
V(x, \lambda r)\leq C\lambda^n V(x,r) \label{e1.2}
\end{equation}
uniformly for all $\lambda\geq 1$ and $x\in X$. In the sequel, we
shall consider $n$ as small as possible. In the Euclidean space with
Lebesgue measure, the parameter $n$ is the dimension of the space.

\medskip

In this paper, we consider the following two conditions
 corresponding to the operator $L$.

 First, the operator $L$ is a non-negative self-adjoint
operator acting on $L^2({X})$ and the semigroup $\{e^{-tL}\}_{t>0}$
generated by $L$ satisfies the {\it Davies-Gaffney condition} (See
for example \cite{D}). That is, there exist constants $C$, $c>0$
such that for any open subsets $U_1,\,U_2\subset X$,
\begin{equation}\label{e1.5}
|\langle e^{-tL}f_1, f_2\rangle| \leq C\exp\Big(-{{\rm
dist}(U_1,U_2)^2\over c\,t}\Big)
\|f_1\|_{L^2(X)}\|f_2\|_{L^2(X)},\quad\forall\,t>0, \nonumber
\leqno{\rm (GE)}
\end{equation}
 for every $f_i\in L^2(X)$ with $\mbox{supp}\,f_i\subset
U_i$, $i=1,2$, where ${\rm dist}(U_1,U_2):=\inf_{x\in U_1, y\in U_2}
d(x,y)$.

\smallskip

 Second, the operator $L$ satisfies
{\it  restriction type estimates}. Given a subset $E\subseteq X$, we
define the projection operator $P_E$ by multiplying by the
characteristic function of $E$, that is,
$$
P_Ef(x)=\chi_E(x) f(x).
$$
For a function $F: {\mathbb R}\rightarrow {\mathbb C}$ and
$R>0$, we denote by $\delta_RF:  {\mathbb R}\rightarrow {\mathbb C}$
the function $x\to F(Rx).$
    Following  \cite{COSY}, we say that a non-negative self-adjoint operator $L$ satisfies
    {\it restriction type estimates} if
  for any $R>0$ and all Borel functions $F$ such that supp $F \subset [0,R]$,
  there exist some $p_0$ and $q$ satisfying $1\leq p_0<2$ and $1\leq q\leq\infty$
  such that
\begin{eqnarray}\label{RE1.4}
\big\|F(\SL)P_{B(x, r)} \big\|_{p_0\rightarrow 2} \leq CV(x,
r)^{{1\over 2}-{1\over p_0}} \big( Rr \big)^{n({1\over p_0}-{1\over
2})}\big\|\delta_RF\big\|_{L^q}
\end{eqnarray}
 for all $x\in X$ and $r\geq 1/R$,   where $n$ is the
dimension entering doubling volume condition (\ref{e1.2}). When
$L=-\Delta$ on $\mathbb{R}^n$, this estimate is equivalent to the
classic $(p, 2)$ restriction estimate of Stein-Tomas. See
\cite{COSY} or Proposition~\ref{prop2.3} below.

\medskip

The aim of this paper is to obtain a sharp H\"ormander-type spectral
multiplier theorem for abstract operators which generate semigroups
satisfying Davies-Gaffney condition. More precisely, our result
shows that restriction type estimates imply sharp spectral
multipiers
 on Hardy spaces $H^p_L(X)$ for $p>0$,
 where  $H^p_L(X)$  is a new class of Hardy spaces   associated to $L$ (
 \cite{ADM, AMR, DL, DP2, DY, DY2, HLMMY, HM, HMMc, JY}, see  Section 2
 below). The theorem is valid for abstract self-adjoint operators. However,
 before the result can be applied one has to verify conditions ${\rm (GE)}$ and \eqref{RE1.4}.
  Usually proving restriction type condition \eqref{RE1.4} is difficult. See discussions in \cite{COSY}.
   We discuss several examples of operators which satisfy
 required restriction type estimates in Section~4. On the other
 hand, condition \eqref{RE1.4} with $p_0=1$ and $q=\infty$ follows from Gaussian estimates
 \eqref{e1.8} for the heat kernel corresponding to the operator. See
 discussions in \cite{COSY} and \cite{DOS}.

Let $\phi$ be a nontrivial compact supported smooth function and
define the Sobolev norm
$\|F\|_{W^{s,q}}=\|(I-d^2/dx^2)^{s/2}F\|_{L^2}$.
  The following theorem is the main result of the paper.

 \begin{theorem}\label{th1.1}
Consider the doubling metric measure space $(X,d,\mu)$ which
satisfies (\ref{e1.2}) with dimension $n$. Assume that the operator
$L$ satifies Davies-Gaffney estimate {\rm (GE)} and  the restriction
type condition \eqref{RE1.4} for some $p_0,q$ satisfying $1\leq
p_0<2$ and $1\leq q\leq \infty$. Suppose that $0<p\leq 1$ and for a
bounded Borel function $F$, there exists some constant
$s>n(1/p-1/2)$ such that
\begin{eqnarray}\label {e1.7}
\sup_{t>0}\|\phi\delta_tF\|_{W^{s,q}}< \infty.
\end{eqnarray}
 Then the operator
$F(\sqrt{L})$ is bounded on $H_L^p(X)$,     i.e., there exists a
constant $C>0$ such that
\begin{eqnarray*}
   \|F(\SL)f\|_{H^p_L(X)}\leq    C\|  f\|_{H^p_L(X)}.
\end{eqnarray*}
\end{theorem}

\noindent {\bf Remarks:}

\smallskip

i) Theorem~\ref{th1.1} is sharp when $q=2$ by considering
Bochner-Riesz means on the spaces $H_L^p(X)$. See
Corollary~\ref{coro3.3} below. When $L=-\Delta$ on $\mathbb{R}^n$,
it satisfies \eqref{RE1.4} with $q=2$ for all $1\leq p_0\leq
(2n+2)/(n+3)$ and from this theorem we can obtain sharp results for
the boundedness of classic Bochner-Riesz means on Hardy spaces
$H^p(\mathbb{R}^n)$.
\smallskip

ii) In  \cite{DY2}, Theorem~\ref{th1.1} was obtained under the
condition (\ref{e1.7}) with the norm $W^{s,\infty}$. Note that for
fixed $p_0$ if condition \eqref{RE1.4} holds for some $q\in [1,
\infty)$,
 then \eqref{RE1.4}  holds for all $\tilde{q}\geq q$ including the case
 $\tilde{q}=\infty$ and also note that the
smaller $q$ is, the weaker condition (\ref{e1.7}) is. Although $q=2$
leads to the sharp result, we have examples that the operator
satisfies \eqref{RE1.4} with some $q>2$ but it does not satisfy
\eqref{RE1.4} with $q=2$. For example, harmonic oscillator
$L=-d^2/dx^2+x^2$ acting on $L^2(\mathbb{R})$ satisfies
\eqref{RE1.4} with $p_0=1$ and $q=4$ other than $q=2$. See
\cite[Section 7.5]{DOS} for more discussion.

\smallskip

iii) We do not expect $q<2$ for condition \eqref{RE1.4}. If
\eqref{RE1.4} holds for some $q<2$, by Theorem~\ref{th1.1}, the
classic Bonchner-Riesz mean operator would be bounded for some
$\delta<n(1/p-1/2)-1/2$ and this contradicts the well known result
that $\delta>n(1/p-1/2)-1/2$ is necessary for the classic
Bonchner-Riesz summability. However, if we consider $p_0 \to r$ norm
in \eqref{RE1.4} with some $r>2$ instead of $p_0 \to 2$ norm, using
analogous argument, we can get similar results for some $1\leq q<2$.

\medskip

A standard application of spectral multiplier theorems is to
consider the boundedness of Bochner-Riesz means. Let us recall that
Bochner-Riesz means of order $\delta$ for a non-negative
self-adjoint operator $L$ are defined by the formula
\begin{equation}\label{e3.8}
S_R^{\delta} (L)  = \bigg(I-{L\over R^2}\bigg)_+^{\delta},\ \ \ \ R>0.
\end{equation}

\noindent In Theorem~\ref{th1.1}, if one chooses $F(\lambda) =
(1-\l^2)^{\delta}_+$ then $F\in W^{\beta,q}$ if and only if
${\delta}>\beta-1/q$. We then have the following corollary, which
generalizes the classical result due to Sj\"olin \cite{Sj} and
Stein-Taibleson-Weiss\cite{STW} on the Bochner-Riesz means, and this
result is sharp for Laplacian on ${\mathbb R}^n$ (see \cite{Sj}).

\begin{coro}\label{coro3.3}
 Assume that the operator
$L$ satifies Davies-Gaffney estimate {\rm (GE)} and  the restriction
type condition \eqref{RE1.4} for some $p_0,q$ satisfying $1\leq
p_0<2$ and $1\leq q\leq \infty$. Suppose that $0<p\leq 1$.
 Then for all  ${\delta}>n(1/p-1/2)-1/q$, we
have
\begin{eqnarray}\label {e3.14}
\Big\|\Big(I-{L\over R^2}\Big)_+^{\delta}\Big\|_{H_L^p\rightarrow
H_L^p}\leq C
\end{eqnarray}
uniformly in $R>0.$
\end{coro}

\medskip

Note that when the semigroup $e^{-tL}$ generated by $L$ has
a kernel $p_t(x,y)$ satisfying a Gaussian upper bound, that is
\begin{eqnarray}\label{e1.8}
\big|p_t(x,y)\big|\leq {C\over V(x,\sqrt{t})} \exp\Big(-{
{d^2(x,y)\over ct}}\Big)
\end{eqnarray}
for all $t>0$,  and $x,y\in X,$  then the Hardy space  $H^p_L(X)$
coincides with  $L^p(X)$ for every $1<p<\infty$ (see
\cite{ADM,HLMMY}). Hence the following corollary is a consequence of
Theorem~\ref{th1.1}.

\begin{coro}\label{coro1.2}
Assume that the heat kernel corresponding to the operator $L$
satisfies {\rm (\ref{e1.8}) }and the operator $L$ satisfies the
restriction type condition \eqref{RE1.4} for some $p_0,q$ satisfying
$1< p_0<2$ and $1\leq q\leq \infty$. Then for any even bounded Borel
function $F$ such that $\sup_{t>0}\|\phi\delta_tF\|_{W^{s,q}}<
\infty$ for some $s>n(1/p_1-1/2)$ and $1\leq p_1\leq p_0$, the
operator  $F(\SL)$ is bounded on $L^p(X)$ for $p_1<p<p_1'$,   i.e.,
there exists a constant $C>0$ such that
\begin{eqnarray*}
   \|F(\SL)f\|_{L^p(X)}\leq    C\|  f\|_{L^p(X)}.
\end{eqnarray*}
\end{coro}

\medskip

The paper is organized as follows. In Section 2, we recall some
preliminary results  about finite speed propagation property,
restriction type estimates and Hardy space $H^p_L(X)$ associated to
an operator $L$, and state a criterion for boundedness of spectral
multipliers on $H^p_L(X)$. In Section 3, we will prove our main
result, Theorem~\ref{th1.1}, by using some estimates for the
operator $F(\SL)$ away from the diagonal  and the restriction type
estimates.

Throughout, the letter ``$C$" and ``$c$" will  denote (possibly
different) constants that are independent of the essential
variables.

\section{Preliminaries}
\setcounter{equation}{0} To simplifying the notation, we shall often
just use $B$ instead of $B(x, r)$. Given $\lambda>0$, we will write
$\lambda B$ for the $\lambda$-dilated ball which is the ball with
the same center as $B$ and with radius $\lambda r.$ For $1\le
p\le+\infty$, we denote the norm of a function $f\in L^p(X,d\mu)$ by
$\|f\|_p$. If $T$ is a bounded linear operator from $%
L^p(X,d\mu)$ to $L^q(X,d\mu)$, $1\le p, \, q\le+\infty$, we write
$\|T\|_{p\to q} $ for the  operator norm of $T$.  Let $\phi$ be a
non-negative $C_c^{\infty}$-function such that
\begin{eqnarray}
{\rm supp}  \phi \subseteq ({1\over 4}, 1) \ \ {\rm and} \ \
\sum_{\ell\in {\Bbb Z}}\phi(2^{-\ell}\lambda)=1\ \ \ {\rm for\ all}\
\lambda>0. \label{e1.6}
\end{eqnarray}

\subsection{Finite speed propagation for the wave
equation.} Following \cite{CS}, we set
\begin{equation*}
\D_\rho=\{ (x,\, y)\in X\times X: {d}(x,\, y) \le \rho \}.
\end{equation*}
Given an operator $T$ from $L^p(X)$ to $L^q(X)$, we write
\begin{equation}\label{e2.1}
\supp \, K_{T} \subseteq \D_\rho
\end{equation}
if $\langle T f_1, f_2 \rangle = 0$ whenever $f_k$ is in~$C(X)$ and
$\supp \, f_k \subseteq B(x_k,\rho_k)$ when $k = 1,2$, and
$\rho_1+\rho_2+\rho < {d}(x_1, x_2)$. One says that $\cos(t\SL)$
satisfies finite speed propagation property if there holds
$$
\supp \, K_{\cos(t\SL)} \subseteq \D_t \quad \forall t\ge 0\,.
\leqno{\rm (FS)}
$$

\noindent More precisely, we have the following result.

 \begin{proposition}\label{prop2.1} Let $L$ be a non-negative self-adjoint operator
acting on $L^2(X)$. Then the finite speed propagation property {\rm
(FS)} and Davies-Gaffney estimate {\rm (GE)} are equivalent.
\end{proposition}

\medskip \noindent {\bf Proof.}\  For the proof, we refer the reader to Theorem~2 in \cite{S}
and Theorem 3.4 in \cite{CS}. See also \cite{CCT}.
\hfill{} $\Box$

\medskip

 The following lemma is a straightforward  consequence of (FS).

\begin{lemma}\label{le2.2}
Assume that $L$ satisfies {\rm (FS)} and that $F$ is an even bounded Borel function with Fourier
transform  $\hat{F}$ satisfying
$\mbox{\rm supp}\; \hat{F} \subset [-\rho,\, \rho]$.
Then
$$
\mbox{\rm supp}\; K_{F(\SL)} \subset \D_\rho.
$$
\end{lemma}

\noindent {\bf Proof.}
If $F$ is an even function, then by the Fourier inversion formula,
$$
F(\SL) =\frac{1}{2\pi}\int_{-\infty}^{+\infty}
  \hat{F}(t) \cos(t\SL) \;dt.
$$
 But supp $\hat{F} \subset [-\rho,\rho]$ and
Lemma~\ref{le2.2} follows from (FS). \hfill{} $\Box$

\subsection {Restriction type estimates.} The following result was obtained in \cite[Proposition~2.3]{COSY}.

\begin{proposition}\label {prop2.4} Suppose that $(X, d, \mu)$ satisfies properties (\ref{e1.2}).
Let  $1\leq p_0<2 $ and   $N>n(1/p-1/2).$ Then condition
\eqref{RE1.4} with $q=\infty$   is equivalent with each of the
following conditions:

\begin{itemize}
\item[(a)]  For all $x>0$ and  $r\geq t>0$    we have
$$
\big\|e^{-t^2L}P_{B(x, r)}\big\|_{p_0\rightarrow 2} \leq CV(x,
r)^{{1\over 2}-{1\over p_0}} \Big({r\over  {t}}\Big)^{n({1\over
p_0}-{1\over 2})}. \leqno{\rm (G_{p_0,2})}
$$

\smallskip

\item[(b)]
For all $x\in X$ and    $r\geq  {t}>0$   we have
$$
\big\|(1+t\SL)^{-N}P_{B(x, r)}\big\|_{p_0\rightarrow 2} \leq CV(x,
r)^{{1\over 2}-{1\over p_0}} \left({r\over  {t}}\right)^{n({1\over
p_0}-{1\over 2})}. \leqno{\rm (E_{p_0,2})}
$$
\end{itemize}
\end{proposition}

Following \cite{GHS}, one says that $L$ satisfies {\it $L^{p_0}$ to
$L^{p'_0} $ restriction estimates} if  the spectral measure
$dE_{\sqrt{L}}(\lambda)$ maps  $L^{p_0}(X)$ to $L^{p'_0}(X)$ for
some $p_0$ satisfying $1\leq p_0\leq {2n/(n+1)}$, with an operator
norm estimate
$$
\big\|dE_{\sqrt{L}}(\lambda)\big\|_{p_0\rightarrow p'_0}\leq C
\lambda^{n({1\over p_0}-{1\over p'_0})-{1}} \leqno{\rm (R_{p_0})}
$$
 for all $\lambda>0$.

\begin{proposition}\label{prop2.3}
Suppose that  there exist   positive constants $0<C_1\leq
C_2<\infty$ such that $
 C_1r^n \leq V(x, r)\leq C_2 r^n$
for every $x\in X $ and $r>0$.  Then  conditions ${\rm (R_{p_0})}$
and \eqref{RE1.4} with $q=2$ are
 equivalent.
\end{proposition}

\begin{proof}
For the proof, we refer the reader to \cite[Proposition 2.4]{COSY}.
\end{proof}

 \medskip

\subsection{Hardy spaces $H_L^p(X)$.} The following definition of Hardy spaces $H_L^p(X)$ comes from
\cite{DY2} (see also \cite{HLMMY}).  Following \cite{AMR}, one can
define the $L^2$ adapted Hardy space
\begin{equation}\label{eq2.H2}
H^2(X) := \overline{R(L)},
\end{equation}
that is, the closure of the range of $L$ in $L^2(X)$.  Then $L^2(X)$
is the orthogonal sum of $H^2(X)$ and the null space $N(L)$.

Consider the following quadratic operators associated to $L$
\begin{eqnarray}
S_{h, K}f(x)=\Big(\int_0^{\infty}\!\!\!\!\int_{\substack{ d(x,y)<t}}
|(t^2L)^{K}e^{-t^2L} f(y)|^2 {d\mu(y)\over V(x,t)}{dt\over
t}\Big)^{1/2}, \quad x\in X \label{e2.3}
\end{eqnarray}
 where $f\in L^2(X)$. For each $K\geq 1$ and $1\leq
p<\infty$, we now define
\begin{eqnarray*}
D_{K, p} =\Big\{ f\in H^2(X): \ S_{h, K}f\in L^p(X)\Big\}, \ \ \ 0<
p<\infty.
\end{eqnarray*}

\begin{definition}{\label{def2.2}} Let $L$ be
 a  non-negative self-adjoint operator on $L^2({X})$
 satisfying the Davies-Gaffney condition {\rm (GE)}. \

(i) For each $0< p\leq 2$, the Hardy space $H^p_{L}(X)$ associated
to $L$  is the completion of the space $D_{1, p}$ in the norm
$$
 \|f\|_{H_{L}^p(X)}=  \|S_{h,1}f\|_{L^p(X)}.
$$

(ii) For each $2<p<\infty$, the Hardy space $H^p_{L}(X)$ associated
to $L$ is the completion of the space $D_{K_0, p}$ in the norm
$$
\|f\|_{H_{L}^p(X)}=  \|S_{h, K_0}f\|_{L^p(X)}, \ \ \ \
K_0=\big[\,{n\over 4}\,\big]+1.
$$
\end{definition}

 Under an assumption of Gaussian upper bounds (\ref{e1.8}), it was proved in \cite{ADM} that $H^p_{L}(X)=L^p(X)$
for all $1<p<\infty$. Note that, in this paper, we only assume the
Davies-Gaffney estimates on the heat kernel of $L$, and hence for
$1<p<\infty$, $p\not= 2$, $H^p_{L}(X)$ may or may not coincide with
the space $L^p(X)$. However, it can be verified that
$H^2_{L}(X)=H^2(X)$ and  the dual of $H^p_{L}(X)$ is
$H^{p'}_{L}(X)$, with $1/p+ 1/p'=1$ (see Proposition 9.4 of
\cite{HLMMY}). We also recall that the $H^p_{L}(X)$ spaces ($1 \le p
< +\infty$) are a family of interpolation spaces for the complex
interpolation method. See  \cite[Proposition~9.5]{HLMMY}).

 \subsection{ A criterion for boundedness of spectral multipliers on $H^p_L(X)$}

 We now state  a criterion from \cite{DY2} that allows us to derive  estimates on
Hardy spaces $H^p_L(X)$. This  generalizes the classical
Calder\'on-Zygmund theory and we would like to  emphasize that the
conditions imposed involve the multiplier operator and its action on
functions but not its kernel.

\begin{lemma}\label{le2.7} Let $L$ be  a  non-negative self-adjoint operator acting on $L^2({X})$
 satisfying the Davies-Gaffney estimate {\rm (GE)}.
 Let $m$ be a bounded Borel function.
 Suppose that $0 < p\leq 1$ and
  $M>{n\over 2}\big({1\over  p}-{1\over 2}\big)$.
 Assume that there exist    constants  $ s>n\big({1\over p}-{1\over 2}\big)$  and $C>0$
  such that for every  $j=2,3\cdots,$

  \begin{eqnarray}
\big\|F(L)(I-e^{-r_B^2L})^M
f\big\|_{L^2(2^jB\setminus2^{j-1}B)}\leq C 2^{-js} \|f\|_{L^2(B)}
\label{e2.4}
\end{eqnarray}

\noindent
 for any  ball $B$ with    radius $r_B$ and for all $f\in L^2(X)$ with supp $f\subset B$.
 Then the operator $F(L)$ extends to a bounded operator on  $H^p_L(X)$. More precisely, there exists
 a constant $C>0$ such that for all $f\in H^p_L(X)$

 \begin{eqnarray}
\|F(L)f\|_{H^p_L(X) }\leq C\|f\|_{H^p_L(X) }.
\label{e33.2}
\end{eqnarray}
 \end{lemma}

\begin{proof}
For the proof, we refer the reader to  \cite[Theorem 3.1]{DY2}.
\end{proof}

\medskip

\section{Proof of Theorem~\ref{th1.1}}
\setcounter{equation}{0}

For $s\in{\mathbb R}$ and $p, q$ in $[1, \infty]$, we denote by  ${B_s^{p, q}}$ the usual Besov space (see, e.g., \cite{BL}).
In order to prove Theorem~\ref{th1.1}, let us first show the following useful auxiliary lemma.

 \begin{lemma}\label{le3.1}
  Assume that operator $L$ satisfies  the
 finite speed propagation property {\rm (FS)}
and condition  \eqref{RE1.4} for some $p_0,q$ satisfying $1\leq
p_0<2$ and $1\leq q\leq \infty$.
 Next assume that  function $F$ is even and supported on
 $[-R,R]$.
 Then for any $s>\max\{n(1/p_0-1/2)-1,0\}$, there exists a constant $C_s$
 such that for any ball $B=B(x,r)$ and for every $j=1,2\cdots$,
\begin{equation}\label {e33.1}
\big\|P_{B(x,2^jr)^c}F(\SL)P_{B(x, r)} \big\|_{p_0\rightarrow 2}
\leq \left \{
\begin{array}{ll}
C_sV(x, r)^{{1\over 2}-{1\over p_0}} \big( Rr \big)^{n({1\over
p_0}-{1\over 2})}(2^jr R)^{-s}\big\|\delta_RF\big\|_{B_s^{q,1}}  &rR\geq 1;\\[8pt]

C_sV(x, R^{-1})^{{1\over 2}-{1\over p_0}} (2^jr
R)^{-s}\big\|\delta_RF\big\|_{B_s^{q,1}}   & rR< 1.
\end{array}
\right.
\end{equation}
\end{lemma}

\begin{proof}
Our approach is inspired by the proof of Lemma 3.4 in \cite{CS}. Fix
$r,j$ and $R$ such that $2^{j-5}r R>1$. Otherwise, by condition
\eqref{RE1.4} the proof of (\ref{e33.1}) is trivial. Note that
$\phi_0$ and $\phi_k$ are smooth even functions supported in
$[-4,4]$ and $[2^k,2^{k+2}]\cup[-2^{k+2},-2^k]$ respectively.
Further, $\phi_0(\l)+\sum_{k\geq 1}\phi_k(\l)=1$ for all $\l$ and
$\phi_0=1$ on $[-2,2]$. We set
$\psi(\lambda)=\phi_0(\lambda/(2^{j-3}r))$ and $\psi_0(\lambda)=
\phi_0(\lambda/(2^{j-3}rR))$. Define $T_\phi$ as $\widehat{T_\phi
F}=\phi\widehat{F}$. Since supp $\psi\subset [-2^{j-1}r,2^{j-1}r]$,
it follows by Lemma~\ref{le2.2},
$$
\mbox{supp}\, K_{T_\psi F(\SL)}\subset \big\{(z,y)\in X\times X:
d(z,y)\leq 2^{j-1}r\big\}.
$$
Hence,
$$
K_{F(\SL)}(z,y)=K_{[F-T_\psi F](\SL)}(z,y)
$$
for all $z,y$ such that $d(z,y)>2^{j-1}r$. We  obtain
\begin{eqnarray}\label {e3.3}
\big\|P_{B(x,2^jr)^c}F(\SL)P_{B(x, r)} \big\|_{p_0\rightarrow 2}
\leq \big\|[F-T_\psi F](\SL)P_{B(x, r)} \big\|_{p_0\rightarrow 2}.
\end{eqnarray}
 Now,
\begin{eqnarray}
F-T_\psi F&=&\sum_{k\geq 0}\delta_{R^{-1}}(\phi_k)[F-T_\psi
F]\nonumber\\
&=&\delta_{R^{-1}}(\phi_0)[F-T_\psi F]-\sum_{k\geq
1}\delta_{R^{-1}}(\phi_k)T_\psi F\nonumber\\
&=&\delta_{R^{-1}}(\phi_0)[F-T_\psi
F]-(1-\delta_{R^{-1}}(\phi_0))T_\psi F\nonumber,
\end{eqnarray}
 since $k\geq 1$, supp $\delta_{R^{-1}}(\phi_k)\subset
[-2^{k+2}R,-2^kR]\cup[2^kR,2^{k+2}R]$, and supp $F\subset [-R,R]$.
It follows that
\begin{eqnarray}\label {e3.4}
\big\|P_{B(x,2^jr)^c}F(\SL)P_{B(x, r)} \big\|_{p_0\rightarrow 2}
&\leq& \big\|\delta_{R^{-1}}(\phi_0)[F-T_\psi F](\SL)P_{B(x, r)}
\big\|_{p_0\rightarrow
2}\nonumber\\
&+&\big\|(1-\delta_{R^{-1}}(\phi_0))T_\psi F(\SL)P_{B(x, r)}
\big\|_{p_0\rightarrow 2}.
\end{eqnarray}

\medskip

\noindent
{\it Case 1: $rR\geq 1$}.

\smallskip

Note that  supp $\delta_{R^{-1}}(\phi_0)\subset [-4R,4R]$. By
condition \eqref{RE1.4},
\begin{eqnarray*}
 \big\|\delta_{R^{-1}}(\phi_0)[F-T_\psi F](\SL)P_{B(x, r)}
\big\|_{p_0\rightarrow 2} &\leq&C V(x, r)^{{1\over 2}-{1\over p_0}}
\big( Rr \big)^{n({1\over
p_0}-{1\over 2})}\big\|\phi_0\delta_{R}[F-T_\psi F]\big\|_{L^q}\\
&\leq&C V(x, r)^{{1\over 2}-{1\over p_0}} \big( Rr \big)^{n({1\over
p_0}-{1\over 2})}\big\|\delta_{R}F-T_{\psi_0} (\delta_{R}F)\big\|_{L^q}\\
&=&C V(x, r)^{{1\over 2}-{1\over p_0}} \big( Rr \big)^{n({1\over
p_0}-{1\over 2})}\big\|\sum_{i\geq
0}T_{\phi_i}[I-T_{\psi_0}]\delta_{R}F\big\|_{L^q}.
\end{eqnarray*}
Note that
$\phi_i(\l)(1-\psi_0(\l))=\phi_i(\l)(1-\phi_0(\l/(2^{j-3}rR)))=0$
for all $\l\in \mathbb{R}$ unless $2^i\geq 2^{j-4}rR$. Consequently,
$T_{\phi_i}[I-T_{\psi_0}]\delta_{R}F=0$ unless $i\geq i_0$, where
$i_0=\log_2(2^{j-4}rR)$, and
\begin{eqnarray}\label {e3.5}
 \big\|\delta_{R^{-1}}(\phi_0)[F-T_\psi F](\SL)P_{B(x, r)}
\big\|_{p_0\rightarrow 2} &\leq&C V(x, r)^{{1\over 2}-{1\over p_0}}
\big( Rr \big)^{n({1\over p_0}-{1\over 2})}\sum_{i\geq
i_0}\big\|T_{\phi_i}[I-T_{\psi_0}]\delta_{R}F\big\|_{L^q}\nonumber\\
&\leq&C V(x, r)^{{1\over 2}-{1\over p_0}} \big( Rr \big)^{n({1\over
p_0}-{1\over 2})}\sum_{i\geq
i_0}\big\|T_{\phi_i}\delta_{R}F\big\|_{L^q}\nonumber\\
&\leq&C V(x, r)^{{1\over 2}-{1\over p_0}} \big( Rr \big)^{n({1\over
p_0}-{1\over 2})}2^{-i_0s}\sum_{i\geq
i_0}2^{is}\big\|T_{\phi_i}\delta_{R}F\big\|_{L^q}\nonumber\\
&\leq&C V(x, r)^{{1\over 2}-{1\over p_0}} \big( Rr \big)^{n({1\over
p_0}-{1\over 2})}(2^jrR)^{-s}\big\|\delta_{R}F\big\|_{B_{s}^{q,1}}.
\end{eqnarray}
We now treat the remain term in formula (\ref{e3.4}). We
claim that for any $s>0$,
\begin{eqnarray}\label {e3.6}
\sup_{\l}(1-\delta_{R^{-1}}(\phi_0)(\l))T_\psi
F(\l)(1+R^{-1}|\l|)^{s+1}\leq C(2^jrR)^{-s}\|\delta_RF\|_{L^q}.
\end{eqnarray}
Let ${\check f}$ denotes the inverse Fourier transform of function
$f$. We observe that $|\l-y|\approx |\l|$ if $|\l|\geq2R$ and
$|y|\leq R$, and hence
\begin{eqnarray*}
&&\hspace{-1.5cm} \sup_{\l}(1-\delta_{R^{-1}}(\phi_0)(\l))T_\psi
F(\l)(1+R^{-1}|\l|)^{s+1}\\
&\leq& \sup_{\l}(1-\phi_0(\l/R))|(\check{\psi}*F)(\l)|(1+|\l|/R)^{s+1}\\
&\leq&
\sup_{\l}(1-\phi_0(\l/R))\Big|\int_{-R}^R F(y)\check{\psi}(\l-y)dy\Big|(1+|\l|/R)^{s+1}\\
&\leq&
\sup_{\l}(1-\phi_0(\l/R))2^{j-3}r\int_{-R}^R |F(y)|(1+2^{j-3}r|\l-y|)^{-s-1}dy(1+|\l|/R)^{s+1}\\
&\leq&
C\sup_{\l}(1-\phi_0(\l/R))2^{j-3}rR(1+2^{j-3}r|\l|)^{-s-1}(1+|\l|/R)^{s+1}\|\delta_R F\|_{L^q}\\
&\leq& C(2^jrR)^{-s}\|\delta_RF\|_{L^q}.
\end{eqnarray*}
 From (\ref{e3.6}) and Proposition~\ref{prop2.4}, it
follows that for any $s>\max\{n(1/p_0-1/2)-1,0\}$,
\begin{eqnarray}\label {e3.7}
&&\hspace{-1.5cm}\big\|(1-\delta_{R^{-1}}(\phi_0))T_\psi F(\SL)P_{B(x, r)}
\big\|_{p_0\rightarrow 2} \nonumber\\
&\leq& \sup_{\l}\left|(1-\delta_{R^{-1}}(\phi_0)(\l))T_\psi
F(\l)(1+R^{-1}|\l|)^{s+1}\right|\big\|(I+R^{-1}\SL)^{-s-1}P_{B(x,
r)}
\big\|_{p_0\rightarrow 2}\nonumber\\
&\leq& CV(x, r)^{{1\over 2}-{1\over p_0}} \big( Rr \big)^{n({1\over
p_0}-{1\over 2})}(2^jr R)^{-s}\big\|\delta_RF\big\|_{L^q}.
\end{eqnarray}

Then estimates of (\ref{e3.4}), (\ref{e3.5}) and (\ref{e3.7}) imply
estimate (\ref{e33.1})  for any $s>\max\{n(1/p_0-1/2)-1,0\}$.

\medskip

\noindent
{\it Case 2: $rR<1$}.

\smallskip

It follows from (\ref{e3.4}) that
\begin{eqnarray}\label {e3.8}
&&\hspace{-1.2cm}\big\|P_{B(x,2^jr)^c}F(\SL)P_{B(x, r)} \big\|_{p_0\rightarrow 2}\nonumber\\
&\leq& \big\|\delta_{R^{-1}}(\phi_0)[F-T_\psi F](\SL)P_{B(x,
R^{-1})} \big\|_{p_0\rightarrow 2} +
\big\|(1-\delta_{R^{-1}}(\phi_0))T_\psi F(\SL)P_{B(x, R^{-1})}
\big\|_{p_0\rightarrow 2}.
\end{eqnarray}
 Replacing $B(x,r)$ by $B(x,R^{-1})$ in (\ref{e3.5}) and
(\ref{e3.7}), a similar argument as in {\it Case 1} shows
(\ref{e33.2}), and we skip it here. The proof of  Lemma~\ref{le3.1}
is complete.
\end{proof}

\bigskip

\noindent {\bf Proof of Theorem~\ref{th1.1}}. To prove  Theorem~\ref{th1.1},  by Lemma~\ref{le2.7}
it suffices to  verify  condition  (\ref{e2.4}).
Recall that $\phi$ is
a non-negative $C_0^{\infty}$ function such that
\begin{eqnarray*}
{\rm supp} \ \phi \subseteq ({1\over 4}, 1) \ \ {\rm and} \ \
\sum_{\ell\in {\Bbb Z}}\phi(2^{-\ell}\lambda)=1 \ \ \ {\rm for\
all}\  \lambda>0.
\end{eqnarray*}
Then
$$
F(\l)=\sum_{\ell\in \mathbb{Z}}\phi(2^{-\ell}\l)F(\l)=\sum_{\ell\in
\mathbb{Z}}F_\ell(\l) \ \ \ {\rm for\ all}\  \lambda>0.
$$
 For every $\ell\in \mathbb{Z}$ and $r>0$, set  $F_{r,M}^\ell=F_\ell(\l)(1-e^{-r^2\l^2})^M$. So for any
ball $B=B(x,r)$,
\begin{eqnarray}\label {e3.9}
\|F(\SL)(I-e^{-r^2L})^M f\|_{L^2(2^jB\backslash 2^{j-1}B)}\leq
\sum_{\ell\in \mathbb{Z}}\|F_{r,M}^{\ell}(\SL)
f\|_{L^2(2^jB\backslash 2^{j-1}B)}.
\end{eqnarray}
Fix $f\in
 L^2(X)$ with supp $f\subset B$ and $j\geq 2$. Note that supp $F_{r,M}^\ell\subset [-2^{\ell},2^{\ell}]$. So
 if $r2^\ell<1$,
 it follows Lemma~\ref{le3.1} that for any
$s>\max\{n(1/p_0-1/2)-1,0\}$
\begin{eqnarray}\label {e3.10}
\|F_{r,M}^{\ell}(\SL) f\|_{L^2(2^jB\backslash 2^{j-1}B)}
&\leq& \|P_{2^jB\backslash
2^{j-1}B}F_{r,M}^{\ell}(\SL)P_{B}\|_{p_0\to
2}\|f\|_{L^{p_0}}\nonumber\\
&\leq&CV(x, 2^{-\ell})^{{1\over 2}-{1\over p_0}} (2^jr
2^\ell)^{-s}\big\|\delta_{2^\ell}F_{r,M}^\ell\big\|_{B_s^{q,1}}\|f\|_{L^{p_0}}\nonumber\\
&\leq&CV(x, 2^{-\ell})^{{1\over 2}-{1\over p_0}} (2^jr
2^\ell)^{-s}\big\|\delta_{2^\ell}F_{r,M}^\ell\big\|_{B_s^{q,1}}V(x,r)^{{1\over p_0}-{1\over 2}}\|f\|_{L^2}\nonumber\\
&\leq&C 2^{-js}(2^\ell
r)^{2M-s}\big\|\phi\delta_{2^\ell}F\big\|_{B_s^{q,1}}\|f\|_{L^2}.
\end{eqnarray}
Similarly if  $r2^\ell\geq 1$, then
\begin{eqnarray}\label {e3.11}
 \|F_{r,M}^{\ell}(\SL) f\|_{L^2(2^jB\backslash 2^{j-1}B)}
&\leq& \|P_{2^jB\backslash
2^{j-1}B}F_{r,M}^{\ell}(\SL)P_{B}\|_{p_0\to
2}\|f\|_{L^{p_0}}\nonumber\\
&\leq& CV(x, r)^{{1\over 2}-{1\over p_0}} \big( 2^\ell r
\big)^{n({1\over
p_0}-{1\over 2})}(2^jr 2^\ell)^{-s}\big\|\delta_{2^\ell}F_{r,M}^\ell\big\|_{B_s^{q,1}}\|f\|_{L^{p_0}}\nonumber\\
&\leq&C\big( 2^\ell r \big)^{n({1\over
p_0}-{1\over 2})}(2^jr 2^\ell)^{-s}\big\|\delta_{2^\ell}F_{r,M}^\ell\big\|_{B_s^{q,1}}\|f\|_{L^2}\nonumber\\
&\leq&C 2^{-js}(2^\ell r)^{n({1\over p_0}-{1\over
2})-s}\big\|\phi\delta_{2^\ell}F\big\|_{B_s^{q,1}}\|f\|_{L^2}.
\end{eqnarray}
Note that for any $\varepsilon>0$
$\|F\|_{B_{s-\varepsilon}^{q,1}}\leq C_\varepsilon \|F\|_{W^{s,q}}$ (see, e.g. \cite{BL}).
Choosing $s$ such that $M>s>n(1/p-1/2)$,  it follows from
(\ref{e1.7}), (\ref{e3.9}), (\ref{e3.10}) and (\ref{e3.11}) that
\begin{eqnarray*}
\|F(\SL)(I-e^{-r^2L})^M f\|_{L^2(2^jB\backslash 2^{j-1}B)}\leq C
2^{-js}\|f\|_{L^2}.
\end{eqnarray*}
This proves (\ref{e2.4}). Hence, by  Lemma~\ref{le2.7}, $F(\SL)$ can be
extended to be
 a bounded operator on $H_L^p(X)$. The proof of Theorem~\ref{th1.1} is complete.
 \hfill{} $\Box$

\bigskip

\noindent {\bf Proof of Corollary~\ref{coro1.2}}. Firstly,  we note
that when $p_1=p_0$, Corollary~\ref{coro1.2} follows from
\cite[Theorem~4.1]{COSY}. Secondly, it follows by
Theorem~\ref{th1.1} that Corollary~\ref{coro1.2} holds for $p_1=1$.
Now we follow an idea as in \cite{M} to  construct a family of
spectral multipliers $\{F_z: z\in \mathbf{C}, 0\leq \mbox{Re}z\leq
1\}$ as follows:
$$
F_z(\l)=\sum_{j=
-\infty}^{\infty}\eta(2^{-j}\l)\Big(1-2^{2j}\frac{d}{d\l^2}\Big)^{(z-\theta)n(1-1/p_0)/2}\big(F(\l)\phi(2^{-j}\l)\big),
$$
where $\theta=(1-1/p_1)/(1-1/p_0)$ and $\eta\in C_c^\infty([1/4,4]),
\phi\in C_c^\infty([1/2,2])$, $\eta=1$ on $[1/2,2]$ and
$\sum_j\eta(2^{-j}\l)= \sum_j\phi(2^{-j}\l)=1$ for all $\l>0$.
Observe that if $z=1+iy$, then
$$
\sup_{t>0}\|\phi\delta_tF_{1+iy}\|_{W^{s_1,q}}\leq
C\sup_{t>0}\|\phi\delta_tF\|_{W^{s,q}}(1+|y|)^{n/2}
$$
 for
some $s_1>n(1/p_0-1/2)$. On the other hand, if $z=iy$, then
$$
\sup_{t>0}\|\phi\delta_tF_{iy}\|_{W^{s_2,q}}\leq
C\sup_{t>0}\|\phi\delta_tF\|_{W^{s,q}}(1+|y|)^{n/2}
$$
for some $s_2>n/2$. It follows by  \cite[Theorem 4.1]{COSY}
 that   $F_{1+iy}(\SL)$ is bounded on
$H_L^{p}(X)$ for $p_0<p<p_0'$ and by Theorem~\ref{th1.1} that
$F_{iy}(\SL)$ is bounded on $H_L^1(X)$. Applying the three line
theorem, we get $F_\theta(\SL)= F(\SL)$ is bounded on $H_L^p(X)$,
that is, $F(\SL)$  is bounded on $L^p(X)$ for $p_1<p<p_1'$.
 \hfill{} $\Box$

\bigskip

\section{Applications}
\setcounter{equation}{0}

Theorem~\ref{th1.1} is valid for abstract self-adjoint operators.
However, before the result can be applied one has to verify
conditions ${\rm (GE)}$ and \eqref{RE1.4}. Usually proving
restriction type condition \eqref{RE1.4} is difficult. See
discussions in \cite{COSY}. In this section, we discuss several
examples of operators which satisfy
 required restriction type estimates and apply our main results to these operators.
\subsection{Sub-Laplacians on homogeneous groups}\label{sec61}

Let ${\bf G}$  be a homogeneous Lie group of polynomial growth with
homogeneous dimension $n$ (see for examples,  \cite{C, DeM, FS, HS})
and let $X_1, ..., X_k$ be a system of
 left-invariant vector fields on ${\bf G}$ satisfying the H\"ormander condition. We define
 the sub-Laplace operator $L$ acting on $L^2({\bf G})$ by the formula
 \begin{eqnarray}
L=-\sum_{i=1}^k X_i^2. \label{e4.1}
\end{eqnarray}

\begin{proposition}\label{prop4.1} Let $L$ be the homogeneous sub-Laplacian
defined by the formula {\rm (\ref{e4.1})} acting on a homogeneous
group ${\bf G}$.  Then condition \eqref{RE1.4} holds for $p_0=1$ and
 $q=2$, and hence results of  Theorem~\ref{th1.1}  and
Corollary~\ref{coro3.3} hold for $q=2$.
\end{proposition}

 \noindent {\bf Proof.}  It is well known that the heat kernel  corresponding
to the operator $L$ satisfies Davies-Gaffeny estimate {\rm (GE)}. It
is also not difficult to check that for some constant $C>0$
$$
\|F(\sqrt L)\|_{L^2(X) \to L^\infty(X)}^2 =C \int_0^\infty |F(t)|^2
t^{n-1}dt.
$$
See for example equation (7.1) of \cite{DOS} or \cite[Proposition
10]{C}. It was proved that the above equality implies condition
\eqref{RE1.4} with $p_0=1$ and $q=2$ (see \cite[Section 12]{COSY}).
Then Theorem~\ref{th1.1} and Corollary~\ref{coro3.3} imply
Proposition~\ref{prop4.1}. \hfill{}$\Box$

\medskip

Proposition~\ref{prop4.1} can be extended to ``quasi-homogeneous''
operators
 acting on homogeneous groups, see \cite{S1} and \cite{DOS}.

\bigskip

\subsection{Shr\"odinger operators on asymptotically conic manifolds}

Asymptotically conic manifolds (see \cite{Me})
 are defined as  the interior of a compact manifold $M$ with boundary, such that
 the metric $g$ is smooth on the interior and in a collar neighbourhood of the boundary
has the form
\[g=\frac{dx^2}{x^4}+\frac{h(x)}{x^2} \]
where $x$ is a smooth boundary defining function  and $h(x)$ is a
smooth family of metrics on the boundary.

\begin{proposition}\label{pprop4.2} Let $(M,g)$ be a nontrapping
asymptotically conic manifold of dimension $n \geq 3$, and let $x$
be a smooth boundary defining function of $\pl M$. Let $L:= -
\Delta+V$ be a Schr\"odinger operator with $V\in x^3C^\infty(M)$ and
assume that $L$ is a positive operator and $0$ is neither an
eigenvalue nor a resonance. Then condition \eqref{RE1.4}
 is true with  $q=2$ for all $1\leq p_0\leq (2n+2)/(n+3)$,
 and hence results of Theorem~\ref{th1.1} and Corollary~\ref{coro3.3} hold for $q=2$.
\end{proposition}

\begin{proof}
It was proved in \cite[Theorem 1.3]{GHS} that condition ${\rm
(R_{p_0})}$
 is satisfied for $L$ when $1\leq p_0\leq (2n+2)/(n+3)$.
 By Proposition~\ref{prop2.3}, Theorem~\ref{th1.1} and
Corollary~\ref{coro3.3}, Proposition ~\ref{pprop4.2} is proved.
\end{proof}
\bigskip

\subsection{Schr\"odinger operators with the inverse-square potential}

In this subsection, we consider Shr\"odinger operators $L=-\Delta +
V$ on $L^2(\RR^n, dx)$, where $V(x) = \frac{c}{| x |^2}$. We assume
that $n > 2$ and $c> -{(n-2)^2/4}$. The classical Hardy inequality
\begin{equation}\label{hardy1}
- \Delta\geq  {(n-2)^2\over 4}|x|^{-2},
\end{equation}
shows that the self-adjoint operator $L$  is non-negative if $c >
-{(n-2)^2/4}$. Set  $p_c^{\ast}=n/\sigma$, $\sigma= \max\{
(n-2)/2-\sqrt{(n-2)^2/4+c}, 0\}$. If $c \ge 0$ then the semigroup
$\exp(-tL)$ is pointwise bounded by the Gaussian upper bound
\eqref{e1.8} and hence acts on all $L^p$ spaces with $1 \le p \le
\infty$.  If $ c < 0$, then $\exp(-tL)$ acts as a uniformly bounded
semigroup on $L^p(\RR^n)$ for $ p \in ((p_c^{\ast})', p_c^{\ast})$
and the range $((p_c^{\ast})', p_c^{\ast})$ is optimal (see for
example \cite{LSV}).

\medskip

For these Shr\"odinger operators, we have the following proposition.
\begin{proposition}{\label{prop4.2}} Assume that $n > 2$ and let   $L=-\Delta + V$ be a Schr\"odinger operator
on $L^2(\RR^n, dx)$ where $V(x) = \frac{c}{| x
|^2}$ and $c>-{(n-2)^2/4}$. Suppose  that
 $p_0\in ((p_c^{\ast}){'}, 2n/(n+2)]$ where $p_c^{\ast}=n/\sigma$ and $\sigma= \max\{ (n-2)/2-\sqrt{(n-2)^2/4+c}, 0\}$
 and $(p_c^{\ast}){'}$ its dual exponent.
Then condition \eqref{RE1.4}
 is true with  $q=2$, and hence results of Theorem~\ref{th1.1} and Corollary~\ref{coro3.3} hold for $q=2$.
\end{proposition}

\begin{proof}
It was proved in \cite[Section 10]{COSY} that $L$ satisfies
restriction estimate ${\rm (R_{p_0})}$
 for all $p_0 \in ((p_c^{\ast})',  \frac{2n}{n+2}]$. If $c \ge 0$, then $p = (p_c^{\ast})' = 1$ is included.
By Proposition~\ref{prop2.3}, ${\rm (R_{p_0})}$ and \eqref{RE1.4}
with $q=2$ are equivalent. Now Proposition~\ref{prop4.2} follows
from Theorem~\ref{th1.1} and Corollary~\ref{coro3.3}.
\end{proof}

 \bigskip

\noindent \noindent {\bf Acknowledgements:} This project was
supported by Australian Research Council Discovery Grant DP
110102488. The author would like to thank X.T. Duong, A. Sikora and
L.X.~Yan for fruitful discussions.

\end{document}